\documentclass[reqno,12pt]{amsart}
\usepackage{amsmath,amsthm,amsfonts,amssymb}
\usepackage[utf8x]{inputenc}
\usepackage[english]{babel}
\usepackage[normalem]{ulem} 

\date{} 
\newtheorem{theorem}{Theorem}[section]

\newtheorem{proposition}[theorem]{Proposition}
\newtheorem{remark}[theorem]{Remark}

\newtheorem{example}[theorem]{Example}
\numberwithin{equation}{section}

\begin{document}

\title[]{On nonlinear Kantorovich problems for cost functions of a special form}
\maketitle

\begin{center}
S.N. Popova~\footnote{Moscow Institute of Physics and Technology; National Research University Higher School of Economics.}
\end{center}

\vskip .2in

{\bf Abstract.}
We study Kantorovich type optimal transportation problems with nonlinear cost functions, including dependence on conditional measures of transport plans. A range of nonlinear Kantorovich problems for cost functions of a special form is considered and
results on existence (or non-existence) of optimal solutions are proved. We also establish the connection between the nonlinear Kantorovich problem with the cost function of some special form and the Monge problem with convex dominance. 

Keywords: optimal transportation problem, nonlinear Kantorovich problem, Monge problem, conditional measure, convex dominance. 

\vskip .2in

\section{Introduction}

Recently new interesting modifications of the classical Kantorovich optimal transportation problem
have been introduced and studied. Kantorovich problem deals with two Radon probability measures $\mu$ and $\nu$
on completely regular  topological spaces $X$ and $Y$ respectively, which are supposed to be Polish spaces in most of the works. 
Measures $\mu$ and $\nu$ are called marginals. 
Consider the set $\Pi(\mu,\nu)$ of all Radon probability measures on the product $X\times Y$ with projections $\mu$ and $\nu$ on the factors. Measures in this set are called Kantorovich plans or transport plans. 
The Kantorovich problem consists in minimization
over the set  $\Pi(\mu,\nu)$ of the integral
$$
J_h(\sigma)=\int_{X\times Y} h(x,y) \, \sigma(dx\, dy)
$$
for a given Borel measurable nonnegative function $h$ on $X\times Y$, called a cost function. 
The cost function $h$ is typically supposed to be continuous
or at least lower semicontinuous in order to guarantee that a minimum exists. The functional $J_h$ is linear
with respect to $\sigma$, so we deal with minimization of a linear functional on a convex set, which is compact in the weak topology. Thereby, typical sufficient conditions for the existence of a minimum reduce to verification of the lower semicontinuity of this functional. 

The Monge problem for the same triple $(\mu, \nu, h)$ consists in finding a Borel mapping
$T \colon X \to Y$ taking $\mu$ into $\nu$, that is $\nu = \mu \circ T^{-1}$, $\mu \circ T^{-1}(B) = \mu(T^{-1}(B))$ for
all Borel sets $B \subset Y$, for which the integral
$$
\int_X h(x, T(x)) \mu(dx)
$$
is minimal. In general, there is only an infimum, denoted by $M_h(\mu,\nu)$.
It can happen that there is no minimum even for nice cost functions $h$. However, if both
measures are Radon, $h$ is continuous,
and $\mu$ has no atoms and is separable,  then $M_h(\mu,\nu)$ coincides with the Kantorovich minimum
(see \cite{BKP}; for Polish spaces this was proved in~\cite{Prat}).
Basic facts about Monge and Kantorovich problems can be found in \cite{AG}, \cite{ABS21},
\cite{BK12}, \cite{Sant}, \cite{V09}.

A modification of the Kantorovich problem was studied in the papers \cite{G}, \cite{Alib}, \cite{BaB}, \cite{Ac}, \cite{BaP}, where the novelty was that the cost function $h$ was allowed to depend also on the plan $\sigma$, that is,
$h$ is a Borel function on $X\times Y\times \mathcal{P}(X\times Y)$, where $\mathcal{P}(X\times Y)$ is the space of Radon
probability measures on $X\times Y$ equipped with the weak topology. Thus, now the functional
$$
J_h(\sigma)=\int_{X\times Y} h(x,y,\sigma) \, \sigma(dx\, dy)
$$
is not linear in $\sigma$. So we call problems of this type {\it nonlinear Kantorovich optimal transportation problems}.
Such a problem can be written as
\begin{equation}\label{nonlinK}
\int_{X\times Y} h(x,y,\sigma) \, \sigma(dx\, dy)\to\min, \quad \sigma\in \Pi(\mu,\nu).
\end{equation}

This research has been continued in \cite{BM22}, \cite{BP22}, \cite{BRez}, \cite{BPR}, \cite{B22}, \cite{B-umn22}. 
In \cite{BRez} it was shown that there exists a minimum in the problem (\ref{nonlinK}) in the case where $h(x,y,\sigma)$ is lower semicontinuous on $X\times Y\times \mathcal{P}(X\times Y)$. Although the functional $J_h$ is not linear anymore, it is still lower semicontinuous with respect to~$\sigma$. 
However, the lower semicontinuity of $h$ with respect to the last argument fails
in the specific situation, where $h$ depends on the measure $\sigma$ through its conditional measures $\sigma^x$ on~$Y$.

For a given measure $\sigma \in \mathcal P(X \times Y)$ with the projection $\mu$ on $X$ 
conditional measures $\sigma^x$, $x \in X$, are Radon probability measures on $Y$ with the property
$$
\int_{X\times Y}\varphi(x)\psi(y)\, \sigma(dx\, dy)
=\int_X \int_Y \varphi(x)\psi(y)\, \sigma^x(dy)\, \mu(dx)
$$
for all bounded Borel functions $\varphi$ on $X$ and $\psi$ on $Y$, where it is also assumed that
the function $x\mapsto \sigma^x(B)$ is $\mu$-measurable for every Borel set $B\subset Y$.
Conditional measures exist under rather broad assumptions, for example, in case of Souslin spaces.
The fact that a measure $\sigma$ with the projection $\mu$ on $X$ has conditional measures
is usually expressed as the equality
$$
\sigma(dx\, dy)=\sigma^x(dy)\, \mu(dx). 
$$

The functional associated with conditional measures
can have the form
$$
\int_X\int_Y  h(x,y,\sigma^x) \, \sigma^x(dy)\, \mu(dx)
$$
or a somewhat simpler form
$$
\int_X h(x,\sigma^x) \, \mu(dx).
$$

Introduce the notation and terminology that will be used in this paper. 
The Borel $\sigma$-algebra of a topological space $X$ will be denoted by $\mathcal{B}(X)$.
Denote by $\mathcal{B}a(X)$ the Baire $\sigma$-algebra on $X$, which is generated by all continuous functions on $X$. 
If $X$ is a completely regular Souslin space, then $\mathcal B(X) = \mathcal{B}a(X)$ (see \cite{B07}). 
A nonnegative Borel measure $\mu$ is called Radon if for every Borel set $B$ and every $\varepsilon>0$
there is a compact set $K\subset B$ with $\mu(B\backslash K)<\varepsilon$.
Recall that on Polish spaces and Souslin spaces (images of Polish spaces under continuous mappings)
all Borel measures are Radon.

The set of all Radon probability measures on $X$ is denoted by $\mathcal{P}(X)$.
A family $M\subset \mathcal{P}(X)$ is called uniformly tight if for every $\varepsilon>0$ there is a compact set $K$ such that
$\mu(X\backslash K)<\varepsilon$ for all $\mu\in M$.

The image of a measure $\mu$ on a measurable space $(X,\mathcal{B})$ under a measurable mapping $F$ to a measurable space
$(E,\mathcal{E})$ is denoted by $\mu\circ F^{-1}$ and defined by the formula
$$
(\mu\circ F^{-1})(E)=\mu(F^{-1}(E)), \quad E\in \mathcal{E}.
$$

The weak topology on $\mathcal{P}(X)$ is induced by the weak topology on the whole space
of signed Radon measures on $X$ generated by the seminorms of the form
$$
\eta \mapsto \biggl|\int_X f\, d\eta\biggr|,
$$
where $f$ is a bounded continuous function on $X$.

For any Radon probability measure $Q$ on the space of measures
$\mathcal{P}(Y)$ with the weak topology the barycenter is defined by the formula
$$
\beta_Q:=\int_{\mathcal{P}(Y)} p\, Q(dp),
$$
where this vector integral with values in the space of measures is understood as the equality
$$
\beta_Q(A)=\int_{\mathcal{P}(Y)} p(A)\, Q(dp)
$$
for all Borel sets $A\subset Y$. It is known that the function
$p\mapsto p(A)$ is Borel measurable on $\mathcal{P}(Y)$ and the measure $\beta_Q$
is $\tau$-additive (see \cite[Proposition~8.9.8 and Corollary~8.9.9]{B07}).
The measure $\beta_Q$ is Radon if and only if
the measure $Q$ is concentrated
on a countable union of uniformly tight compact sets in $\mathcal{P}(Y)$ (see \cite[Proposition~3.1]{B-umn22}).
A practically useful sufficient condition is this: the space $Y$ is Souslin completely regular,
because in this case the space $\mathcal{P}(Y)$ is also Souslin (see \cite{B18}), hence all Borel measures on it are Radon.

If $P$ is a Radon measure on $X\times \mathcal{P}(Y)$, $\mu$ is its projection on~$X$ and
there are conditional
measures $P^x$ on~$\mathcal{P}(Y)$ with respect to~$\mu$, then the barycenter of the projection $P_{\mathcal{P}}$ of $P$ on $\mathcal{P}(Y)$
is given by
$$
\beta_{P_{\mathcal{P}}}(B)=\int_X\int_{{\mathcal{P}(Y)}}  p(B)\, P^x(dp)\, \mu(dx).
$$

The goal of this paper is to study the nonlinear Kantorovich problems for a variety of special classes of cost functions. 
Section~2 addresses the nonlinear Kantorovich problem with the cost function $h(x, y, \sigma^x)$. We show that the existence of a solution in this problem does not hold if the function $h$ is continuous and the function $p \mapsto h(x, y, p)$ is convex in $p$ for all $x \in X, y \in Y$. In Section~3 we prove a new existence result in the nonlinear problem with conditional
measures, where we consider cost functions of the form $h(x, p) = \Psi(g(p) - f(x))$. 
Section~4 concerns the nonlinear Kantorovich problem with the cost function $h(x, g(\sigma^x))$, 
where $g(p) = \int_Y F(y) p(dy)$ and $F \colon Y \to \mathbb R^d$ is a Borel measurable function. We show the connection of
the nonlinear problem with the cost function $h(x, g(\sigma^x))$ and a certain Monge problem with convex dominance. 
Section~4 also contains examples of non-existence of minimizers for the nonlinear problem with the cost function $h(x, g(\sigma^x))$. Moreover, we consider the generalization of the nonlinear Kantorovich problem with the cost function $h(x, g(\sigma^x))$ 
to the infinite-dimensional case, where $g(p) = \int_Y F(y) p(dy)$, $F \colon Y \to U$ is a Borel measurable mapping and $U$ is a separable Banach space. 

\section{Non-existence for the Kantorovich problem with the cost function $h(x, y, \sigma^x)$}

Consider the nonlinear Kantorovich problem with conditional measures
\begin{equation}\label{xp}
\int_{X} h(x, \sigma^x) \mu(dx) \to \inf, \quad \sigma \in \Pi(\mu, \nu), \ \sigma(dx dy) = \sigma^x(dy) \mu(dx). 
\end{equation}

In \cite{BaB} it has been proved that if $X, Y$ are Polish spaces, the function $h \colon X \times \mathcal P(Y) \to \mathbb R$ is lower semicontinuous on $X \times \mathcal P(Y)$ and convex in $p$, then there exists a minimum in the problem~(\ref{xp}). 

In \cite{BRez} the existence theorem for the nonlinear Kantorovich problem has been generalized to the case of completely regular topological spaces. 

\begin{theorem}[\cite{BRez}] \label{th_xp} 
Let $X, Y$ be completely regular topological spaces, let $\mu$ and $\nu$ be Radon probability measures on $X$ and $Y$ respectively. Suppose that the function $h \colon X \times \mathcal P(Y) \to [0, +\infty)$ is measurable relative to $\mathcal Ba(X) \times \mathcal Ba(\mathcal P(Y))$, lower semicontinuous on the sets of the form $K \times S$, where $K \subset X$ is a compact and $S \subset P(Y)$ is uniformly tight, and the function $p \mapsto h(x, p)$ is convex in $p$ for all $x \in X$. Then 
 there exists a minimum in the problem~(\ref{xp}). 
\end{theorem}

Consider the nonlinear Kantorovich problem
\begin{equation}\label{xyp}
\int_{X \times Y} h(x, y, \sigma^x) \sigma(dx dy) \to \inf, \quad \sigma \in \Pi(\mu, \nu), \ \sigma(dx dy) = \sigma^x(dy) \mu(dx),
\end{equation}
where $h \colon X \times Y \times \mathcal P(Y) \to \mathbb R$ is a Borel measurable function. 

We show that for the cost function of the form $h(x, y, \sigma^x)$ such an existence theorem fails. The nonlinear problem~(\ref{xyp}) may not have a minimizer if $h$ is a bounded continuous function on $X \times Y \times \mathcal P(Y)$ (or, moreover, Lipschitz) and is convex in $p$. We show that this can occur even if both marginals coincide with Lebesgue measure on the unit interval. 

Below we need the following notation. 
The standard Kantorovich--Rubinshtein norm on Radon measures on the metric space $X$ is defined by the formula
$$
\|\sigma\|_{KR}=\sup_{f\in {\rm Lip}_1, \, |f|\le 1} \int_X f\, d\sigma,
$$
where ${\rm Lip}_1$ is the class of all $1$-Lipschitz functions on~$X$. This norm generates the weak topology
on the subset of nonnegative measures (see \cite[Theorem~8.3.2]{B07}).

\begin{theorem}\label{th_xyp}
Let $X = Y = [0, 1]$ and let $\mu = \nu = \lambda$ be Lebesgue measure on the interval $[0, 1]$. 
There exists a bounded continuous function $h \colon X \times Y \times \mathcal P(Y) \to \mathbb R$ (moreover, $h$ is Lipschitz when $\mathcal P(Y)$  is considered with the Kantorovich--Rubinshtein norm), such that
the function $p \mapsto h(x, y, p)$ is convex in $p$ for all $x \in X, y \in Y$, and the nonlinear Kantorovich problem~(\ref{xyp}) has no minimizer. 
\end{theorem}

\begin{proof}
For every $x \in [0, 1]$ take the following two probability measures on $[0, 1]$:
$$
\nu^1_x(dy) = 2I_{[0, (1+x)/4] \cup [(3+x)/4, 1]}\, dy, \quad \nu^2_x(dy) = 2I_{[(1+x)/4, (3+x)/4]} \, dy,
$$
where $I_S$ is the indicator function of a set $S$. The mappings $x\mapsto \nu^1_x$ and $x\mapsto \nu^2_x$ are $1$-Lipschitz in variation, hence also in the Kantorovich--Rubinshtein norm.

Set 
$$
h_1(x, p) = \min_{t \in [0, 1]} \|p - (t \nu^1_x + (1-t) \nu^2_x)\|_{KR}.
$$ 
The function $h_1(x, p)$ is convex in $p$ and is $1$-Lipschitz in every argument separately, hence is jointly Lipschitz.

Let $\varphi_1, \varphi_2 \colon [0, 1] \to \mathbb R$ be nonnegative Lipschitz functions with supports in the intervals  
$[0, 1/4]$ and $[1/2, 3/4]$ respectively, such that $\int_0^1 \varphi_i(y) dy = 1/2$ for any $i \in \{1, 2\}$. 
Set 
$$
g_i(p) = \int_0^1 \varphi_i(y) p(dy), \quad \mbox{where  } i \in \{1, 2\}.
$$ 
Then we have $g_i(\nu^j_x) = \delta_i^j$ for every $x \in [0, 1]$ (where $\delta_i^j  = 1$ if $i = j$ and $\delta_i^j  = 0$ if $i \neq j$, $i, j \in \{1, 2\}$), since for the restrictions of measures $\nu^j_x$ to $[0, 1/4]$ and $[1/2, 3/4]$ we have 
$\nu^1_x|_{[0, 1/4]} = 2dy$, $\nu^2_x|_{[1/2, 3/4]} = 2dy$ and the measures $\nu^1_x$ and $\nu^2_x$ are mutually singular for every $x \in [0, 1]$. 
 
Set 
$$
h_2(y, p) = \varphi_2(y) g_1(p), \qquad h(x, y, p) = h_1(x, p) + h_2(y, p).
$$ 
The function $h_2(y, p)$ is jointly Lipschitz and linear in $p$, hence the function $h(x, y, p)$ is jointly Lipschitz and convex in $p$. 
We have 
$$
\int_{X \times Y} h_2(y, \sigma^x) \sigma(dx dy) = \int_X \int_Y h_2(y, \sigma^x) \sigma^x(dy) \mu(dx) = \int_X g_1(\sigma^x) g_2(\sigma^x) \mu(dx)
$$
and
$$
\int_{X \times Y} h(x, y, \sigma^x) \sigma(dx dy) = \int_X h_1(x, \sigma^x) \mu(dx) + \int_X g_1(\sigma^x) g_2(\sigma^x) \mu(dx). 
$$

We show that the infimum in this problem is zero, but it is not attained.
Suppose that there exists a measure $\sigma \in \Pi(\mu, \nu)$ with
$$
\int_{X \times Y} h(x, y, \sigma^x) \sigma(dx dy) = 0. 
$$ 
Then $h_1(x, \sigma^x) = 0$ for almost every $x$, hence
$$
\sigma^x = t(x) \nu^1_x + (1-t(x)) \nu^2_x
$$ 
for almost every $x$, where $t(x) \in [0, 1]$. In this case we obtain $g_1(\sigma^x) = t(x)$, $g_2(\sigma^x) = 1 - t(x)$. 
Therefore, 
$$
\int_{X \times Y} h(x, y, \sigma^x) \sigma(dx dy) = \int_X t(x) (1 - t(x)) \mu(dx). 
$$ 
This implies that $t(x) \in \{0, 1\}$ for almost every $x$, i.e. $\sigma^x \in \{\nu^1_x, \nu^2_x\}$ for almost every $x$. 
As we have shown in \cite{BPR}, such a measure $\sigma \in \Pi(\mu, \nu)$ does not exist. 

Let us show that there exists a sequence of measures $\sigma_n \in \Pi(\mu, \nu)$ with
$$
\int_{X \times Y} h(x, y, \sigma_n^x) \sigma_n(dx dy) \to 0.
$$
Set
$$
\sigma_n^x = \nu^1_{(2k+1)/2^n} \quad \hbox{ if }\ x \in [(2k)/2^n, (2k+1)/2^n),
$$
$$
\sigma_n^x = \nu^2_{(2k+1)/2^n} \quad \hbox{ if } \ x \in [(2k+1)/2^n, (2k+2)/2^n),
$$
where $k \in \{0, \ldots, 2^{n-1} - 1\}$,
and take the measures $\sigma_n(dx\, dy) = \sigma_n^x(dy)\, dx$. Then $\sigma_n \in \Pi(\mu, \nu)$, 
because the projection on the first factor is obviously Lebesgue measure and the same is true for the second factor, 
i.e., the integral of $\sigma_n^x$ over $[0,1]$ is Lebesgue measure, since 
$$
\int_{(2k)/2^n}^{(2k+2)/2^n} \sigma_n^x \, dx = \frac{1}{2^n} (\nu^1_{(2k+1)/2^n} + \nu^2_{(2k+1)/2^n})
 = \frac{1}{2^{n - 1}}\, dy
$$
for all $k \in \{0, \dots, 2^{n - 1} - 1\}$.
By construction 
\begin{multline*}
\int_0^1 h_1(x, \sigma_n^x)\, dx \le \sum_{k = 0}^{2^{n-1} - 1}
\int_{(2k)/2^n}^{(2k+1)/2^n} \|\nu^1_{(2k+1)/2^n} - \nu^1_x\|_{KR}\, dx
\\ + \sum_{k = 0}^{2^{n-1} - 1} \int_{(2k+1)/2^n}^{(2k+2)/2^n} \|\nu^2_{(2k+1)/2^n} - \nu^2_x\|_{KR}\, dx \le 1/2^n. 
\end{multline*}
Moreover, $$\int_{X \times Y} h_2(x, \sigma_n^x) \sigma_n(dx dy) = 0,$$
since $g_1(\sigma_n^x) g_2(\sigma_n^x) = 0$ for every $x \in [0, 1]$. 
Therefore, $\int_{X \times Y} h(x, y, \sigma_n^x) \sigma_n(dx dy) \to 0$. 

Thus, the infimum in the nonlinear problem~(\ref{xyp}) is not attained.

\end{proof}

\section{An existence result for the Kantorovich problem with the cost function $h(x, \sigma^x)$}

Consider the nonlinear Kantorovich problem 
\begin{equation}\label{psi}
\int_X h(x, \sigma^x) \mu(dx) \to \inf, \quad \sigma \in \Pi(\mu, \nu), \ \sigma(dx dy) = \sigma^x(dy) \mu(dx),
\end{equation}
for the cost function of the form $h(x, p) = \Psi(g(p) - f(x))$, where $f \colon X \to \mathbb R^d$, $g \colon \mathcal P(Y) \to \mathbb R^d$ and $\Psi \colon \mathbb R^d \to \mathbb R$.  

\begin{theorem}\label{th_psi}
Let $X$ and $Y$ be Souslin spaces, $\mu \in \mathcal P(X)$, $\nu \in \mathcal P(Y)$.  
Suppose that $h(x, p) = \Psi(g(p) - f(x))$, where $f \colon X \to \mathbb R^d$ is a continuous function, 
$g \colon \mathcal P(Y) \to \mathbb R^d$ is a bounded continuous function,  
$\Psi \colon \mathbb R^d \to \mathbb R$ is a nonnegative strictly convex function, 
$\Psi \in C^{1, 1}_{loc}(\mathbb R^d)$, that is, $\Psi$ is differentiable and its gradient is locally Lipschitz. 
Suppose that $\mu(f^{-1}(A)) = 0$ for every set $A \subset \mathbb R^d$ with  Hausdorff dimension at most $d - 1$, if $d > 1$, and $\mu$ is atomless, if $d = 1$. 
Assume that one of the following holds:

{\rm(i)} the sets $\{p \in \mathcal P(Y): g(p) = c\}$ are convex for all $c \in \mathbb R^d$,

{\rm(ii)} the sets $\{p \in \mathcal P(Y): g(p) \le c\}$ are convex for all $c \in \mathbb R^d$ and
the function $\Psi$ is increasing, where we consider the coordinatewise order on $\mathbb R^d$: for two vectors $u = (u_1, \dots, u_d), v = (v_1, \dots, v_d) \in \mathbb R^d$ we have $u \le v$ if $u_i \le v_i$ for all $i \in \{1, \dots, d\}$,
and $\Psi(u) \le \Psi(v)$ for all $u, v \in \mathbb R^d$, $u \le v$. 

Then there exists a minimum in the problem~(\ref{psi}). 
\end{theorem}

For the proof of Theorem \ref{th_psi} we consider the linear Kantorovich problem with a fixed barycenter $\beta \in \mathcal P(Y)$:
\begin{equation}\label{Kfix}
\int_{X\times \mathcal{P}(Y)} h(x, p)\, P(dx\, dp)\to \min, \quad  P \in \Pi^\beta(\mu), 
\end{equation}
where $\Pi^\beta(\mu)$ is the set of all Radon probability measures $P$
on $X\times \mathcal{P}(Y)$ such that the projection $P_X$ of $P$ on $X$ is $\mu$ and
the barycenter of the projection $P_{\mathcal{P}}$ of $P$ on $\mathcal{P}(Y)$ is a given
measure $\beta\in \mathcal{P}(Y)$:
$$
\Pi^\beta(\mu):=\{P \in \mathcal{P}(X\times \mathcal{P}(Y))\colon P_X=\mu, \quad \beta_{P_{\mathcal{P}}}=\beta\}.
$$

Let us recall that a set $\Gamma\subset X\times Z$  is called strongly $h$-monotone for a function
$h$ on $X\times Z$ if there exist Borel measurable functions $\varphi \colon X \to [-\infty, \infty)$ and $\psi \colon Z \to [-\infty, \infty)$ such that $\varphi(x) + \psi(z) \le h(x, z)$ for all $x \in X, z \in Z$ and $\varphi(x) + \psi(z) = h(x, z)$ for all $(x, z) \in \Gamma$. 

It is known (see \cite{BeigGMS}) that if $X$ and $Z$ are Souslin spaces, $\mu\in \mathcal{P}(X)$,
 $\nu\in \mathcal{P}(Z)$,  and for a Borel measurable cost function $h$ there is an optimal measure
 $\sigma\in \Pi(\mu,\nu)$, then $\sigma$ is concentrated on a Borel strongly $h$-monotone set.
 We shall apply this result in the case where $Z=\mathcal{P}(Y)$ and $Y$ is Souslin. 

For a concave function $f$ on an open convex set $U \subset \mathbb R^d$ we denote by $\partial f(x_0)$ the superdifferential of the function $f$ at the point $x_0 \in U$:
$$\partial f(x) = \{v \in \mathbb R^d: f(x) - f(x_0) \le \langle x - x_0, v \rangle \quad \forall x \in U\},$$
where $\langle \cdot, \cdot \rangle$ denotes the Euclidean scalar product in $\mathbb R^d$. 

We use the following proposition about the solutions of the problem (\ref{Kfix}) proved in \cite{BPR}. 

\begin{proposition}[\cite{BPR}] \label{prop1}
{\rm(i)} Let $X, Y$ be completely regular topological spaces. Let $h$ be a lower semicontinuous function on the product $X\times \mathcal{P}(Y)$.
For all $\mu\in \mathcal{P}(X)$ and $\beta\in \mathcal{P}(Y)$ the Kantorovich problem (\ref{Kfix})
with a fixed barycenter has a solution.

{\rm(ii)} Every optimal measure $P$ for this problem is also optimal for the classical linear problem with the same cost function
and marginals $\mu$ and~$P_{\mathcal{P}}$, where $P_{\mathcal{P}}$ is the projection of $P$ on  $\mathcal{P}(Y)$.

{\rm(iii)} Finally, if $X$ and $Y$ are Souslin spaces, then $P$ is concentrated on a strongly $h$-monotone  set.
\end{proposition}

\begin{proof}[Proof of Theorem \ref{th_psi}] 
Let  us consider the Kantorovich problem (\ref{Kfix})  on $X \times \mathcal{P}(Y)$ with a fixed barycenter equal to $\nu$.
According to Proposition~\ref{prop1}, there is a point of minimum
$P \in \Pi^\nu(\mu)\subset \mathcal P(X \times \mathcal P(Y))$.
Moreover, the measure $P$ is concentrated on a Borel strongly $h$-monotone set $\Gamma$.
Let $\varphi \colon X \to [-\infty, \infty)$ and $\psi \colon \mathcal P(Y) \to [-\infty, \infty)$ be Borel measurable functions such that $\varphi(x) + \psi(p) \le h(x, p)$ for all $x \in X$, $p \in \mathcal P(Y)$, and
$\varphi(x) + \psi(p) = h(x, p)$ for all $(x, p) \in \Gamma$.  
Let $(x_0, p_0) \in \Gamma$. Then 
$$
\varphi(x_0) + \psi(p_0) = \Psi(g(p_0) - f(x_0))
$$ 
and
$$
\varphi(x) \le \inf_{p \in \mathcal P(Y)} (h(x, p) - \psi(p)) \quad \forall x \in X,
$$
that is, $\varphi(x) \le \chi(f(x))$ for all $x \in X$, where
$$
\chi(v)  = \inf_{p \in \mathcal P(Y)} (\Psi(g(p) - v) - \psi(p)) \quad \forall v \in \mathbb R^d.
$$ 
Let $V = \{v \in \mathbb R^d: \chi(v) \neq -\infty\}$. Note that $f(x) \in V$ for every $x \in \Gamma_X$, where $\Gamma_X = \{x \in X: \exists p \in \mathcal P(Y), (x, p) \in \Gamma\}$ is the projection of $\Gamma$ on $X$. We have $\mu(\Gamma_X) = 1$, since $P(\Gamma) = 1$, and hence $\mu(f^{-1}(V)) = 1$. 

Denote $v_0 = f(x_0)$. Fix some $\delta > 0$. Denote by $U(v_0, \delta)$ the open ball with center $v_0$ and radius $\delta$.  
For some $R > 0$ we have $|g(p)| \le R$ for all $p \in \mathcal P(Y)$. 
Since $\Psi \in C^{1, 1}_{loc}(\mathbb R^d)$, there exists $\lambda > 0$ such that $\Psi(v) - \lambda |v|^2$ is a concave function in the ball $U(0, |v_0| + R + \delta)$. Then the function
$v \mapsto \Psi(g(p) - v) - \lambda |v|^2$ is concave in the ball $U(v_0, \delta)$ for every $p \in \mathcal P(Y)$. Thus the function
$$\chi(v) - \lambda |v|^2 = \inf_{p \in \mathcal P(Y)} (\Psi(g(p) - v) - \lambda |v|^2 - \psi(p))$$ 
is concave in the ball $U(v_0, \delta)$ (as infimum of concave functions).
We have 
$$
\chi(f(x_0)) + \psi(p_0) = \Psi(g(p_0) - f(x_0))
$$ 
and 
$$
\chi(v) + \psi(p_0) \le \Psi(g(p_0) - v) \quad \forall v \in \mathbb R^d.
$$ 
Therefore, 
$$
\chi(v) - \chi(v_0) \le \Psi(g(p_0) - v) - \Psi(g(p_0) - v_0) \quad \forall v \in U(v_0, \delta). 
$$
This implies the inclusion
$$
\partial (\Psi(g(p_0) - v) - \lambda |v|^2)|_{v = v_0} \subset \partial (\chi(v) - \lambda |v|^2)|_{v = v_0}. 
$$
for superdifferentials of concave functions. 
Since the function $\chi(v) - \lambda |v|^2$ is concave in the ball $U(v_0, \delta)$, 
the set 
$$\{v \in U(v_0, \delta): \chi(v) \neq -\infty\} = U(v_0, \delta) \cap V$$ is convex. 
Let $V(v_0)$ be the interior of the set $U(v_0, \delta) \cap V$, then $(U(v_0, \delta) \cap V) \setminus V(v_0)$ has Hausdorff dimension at most $d - 1$. 
It is known that a concave function on an opex convex subset of $\mathbb R^d$ is differentiable out of a set with Hausdorff dimension at most $d - 1$. 
Let $V' \subset V$ be the set of differentiability points of the function $\chi$, then $V \setminus V'$ has Hausdorff dimension at most $d - 1$. Assume that $f(x_0) \in V'$. Then we obtain that 
$$\nabla \Psi(g(p_0) - f(x_0)) = \nabla \chi(f(x_0)).$$ 
Since the function $\Psi$ is strictly convex, the mapping 
$z \mapsto \nabla \Psi(z)$ is injective. This implies that the value $g(p_0)$ is uniquely determined, i.e. 
$g(p_1) = g(p_2)$ for all $p_1, p_2 \in \mathcal P(Y)$ such that $(x_0, p_1), (x_0, p_2) \in \Gamma$. 

Consider separately two cases: $d > 1$ and $d = 1$. 
First, let $d > 1$. Set $X_0 = f^{-1}(V')$. 
Then $\mu(X \setminus X_0) = \mu(f^{-1}(V \setminus V')) = 0$ because the set $V \setminus V'$ has Hausdorff dimension at most $d - 1$.
Set 
$$
\sigma^x = \int_{\mathcal P(Y)} p P^x(dp), \, x \in X,
$$ 
and $\sigma(dx\, dy) = \sigma^x(dy) \mu(dx)$. 
Then 
$$
\int_X \sigma^x\, \mu(dx) = \int_{\mathcal{P}(Y)} p \, P(dx\, dp) = \nu,
$$
i.e. $\sigma \in \Pi(\mu, \nu)$.
It follows from the above that there exists a function $G \colon X \to \mathbb R^d$ such that
$$P^x(p \in \mathcal P(Y): g(p) = G(x)) = 1 \quad \mbox{for $\mu$-a.e. $x$}.$$ 

In the case \rm{(i)} we have $g(\sigma^x) = G(x)$ for $\mu$-a.e. $x$, since the set $\{p \in \mathcal P(Y): g(p) = G(x)\}$ is convex and closed. 
Therefore, 
\begin{multline*}
\int_X \Psi(g(\sigma^x) - f(x)) \, \mu(dx) = \int_X \Psi(G(x) - f(x)) \, \mu(dx) = \\
= \int_X \int_{\mathcal P(Y)} \Psi(g(p) - f(x)) P^x(dp) \, \mu(dx) = \int_{X \times \mathcal P(Y)} \Psi(g(p) - f(x)) \, P(dx dp). 
\end{multline*}
In the case \rm{(ii)} we have $g(\sigma^x) \le G(x)$ for $\mu$-a.e. $x$, since the set $\{p \in \mathcal P(Y): g(p) \le G(x)\}$ is convex and closed. Hence $\Psi(g(\sigma^x) - f(x)) \le \Psi(G(x) - f(x))$, since $\Psi$ is increasing. 
Therefore,
\begin{multline*}
\int_X \Psi(g(\sigma^x) - f(x)) \, \mu(dx) \le \int_X \Psi(G(x) - f(x)) \, \mu(dx) = \\ =
\int_{X \times \mathcal P(Y)} \Psi(g(p) - f(x)) \, P(dx dp).
\end{multline*}

Thus in both cases \rm{(i)} and \rm{(ii)} we obtain that the measure $\sigma$ delivers a minimum in the problem (\ref{psi}). 

Let $d = 1$. Then the set $V \setminus V'$ is at most countable. Let $V \setminus V' = \{c_j\}$.
Let $X_j = f^{-1}(c_j)$ and let  $\pi_j$ be the projection on $\mathcal{P}(Y)$ of the restriction of the measure
$P$ to the set $X_j \times \mathcal{P}(Y)$.
We take into account only $j$ with $\mu(X_j)>0$.
For every such $j$, since $\mu$ has no atoms,  there is a Borel mapping
$$
T_j \colon X_j \to \mathcal{P}(Y)
$$
such that
$$
\mu|_{X_j} \circ T_j^{-1} = \pi_j,
$$
see, e.g., \cite[Theorem~9.1.5]{B07}.
We now set
$$
\sigma^x = T_j(x) \quad \hbox{whenever } \ x \in X_j.
$$
Then
$$
\int_{X_j} \sigma^x\, \mu(dx) = \int_{\mathcal{P}(Y)} p \, \pi_j(dp), 
$$
\begin{multline*}
\int_{X_j} \Psi(g(\sigma^x) - c_j)\, \mu(dx) = \int_{\mathcal{P}(Y)} \Psi(g(p) - c_j) \, \pi_j(dp) =  
\int_{X_j \times Y} \Psi(g(p) - f(x)) P(dx dp). 
\end{multline*}
Let $X_0 = f^{-1}(V')$. Set 
$$\sigma^x = \int_{\mathcal P(Y)} p P^x(dp), \ x \in X_0.$$
Then similarly as in the case $d > 1$ we obtain
$$
\int_{X_0} \Psi(g(\sigma^x) - f(x)) \mu(dx) \le \int_{X_0 \times \mathcal P(Y)} \Psi(g(p) - f(x)) P(dx dp).
$$
Therefore,
$$
\int_X \sigma^x\, \mu(dx) = \int_{\mathcal{P}(Y)} p \, P(dx\, dp) = \nu,
$$
$$
\int_X \Psi(g(\sigma^x) - f(x)) \, \mu(dx) \le \int_{X\times {\mathcal{P}(Y)}} \Psi(g(p) - f(x)) \, P(dx\, dp).
$$ 

\end{proof}

The class of cost functions covered by Theorem \ref{th_psi} contains the so called barycentric cost functions of the form
$h(x, p) = \Psi(b(p) - x)$, where $X = Y = \mathbb R^d$, $\Psi \colon \mathbb R^d \to \mathbb R$ is a strictly convex function 
and $b(p)$ denotes the barycenter of a measure $p \in \mathcal P(\mathbb R^d)$. In this case the existence of a minimum in the  nonlinear problem (\ref{th_psi}) follows from Theorem \ref{xp}, since the function $p \mapsto h(x, p)$ is convex in $p$. However, in the statement of Theorem \ref{th_psi} the function $h(x, p) = \Psi(g(p) - f(x))$ need not be convex in $p$ if $g$ is not linear. Consider the following example. 

\begin{example}
\rm
The function 
$$g(p) = G\Bigl(\int_Y \varphi_1(y) p(dy), \dots, \int_Y \varphi_k(y) p(dy)\Bigr),$$ 
where $\varphi_1, \dots, \varphi_k \in C_b(Y)$ and $G \colon \mathbb R^k \to \mathbb R^d$ is an injective continuous function, satisfies the condition that
the sets $\{p \in \mathcal P(Y): g(p) = c\}$ are convex for all $c \in \mathbb R^d$. 
This follows by convexity of the sets 
$$\Bigl\{p \in \mathcal P(Y): \int_Y \varphi_1(y) p(dy) = a_1, \dots, \int_Y \varphi_k(y) p(dy) = a_k\Bigr\}$$ 
for all $a_1, \dots, a_k \in \mathbb R$.  
\end{example}

\begin{example}
\rm
The function 
$$g(p) = \Bigl(\Bigl(\int_Y \varphi_1(y) p(dy)\Bigr)^{k_1}, \dots, \Bigl(\int_Y \varphi_d(y) p(dy)\Bigr)^{k_d}\Bigr),$$ 
where $\varphi_1, \dots, \varphi_d \in C_b(Y)$ and $k_1, \dots, k_d \in \mathbb N$, satisfies the condition that
the sets $\{p \in \mathcal P(Y): g(p) \le c\}$ are convex for all $c \in \mathbb R^d$. 
\end{example}

\section{Kantorovich problem with the cost function $h(x, g(\sigma^x))$}

Consider the nonlinear Kantorovich problem 
\begin{equation}\label{gp}
\int_X h(x, g(\sigma^x)) \mu(dx) \to \inf, \quad \sigma \in \Pi(\mu, \nu), \ \sigma(dx dy) = \sigma^x(dy) \mu(dx),
\end{equation}
for the cost function of the form $h(x, g(p))$, where $g \colon \mathcal P(Y) \to \mathbb R^d$ and
$h \colon X \times \mathbb R^d \to \mathbb R$ are Borel measurable functions.
In the general case the minimization problem (\ref{gp}) turns out to be rather complicated and may not attain a minimum. Indeed, even for the product-type cost function $h(x, g(p)) = f(x)g(p)$, where $f \colon X \to \mathbb R$ and $g \colon \mathcal P(Y) \to \mathbb R$ are bounded continuous functions, the nonlinear problem (\ref{gp}) may not have a minimizer, as we have shown in \cite{BPR}. 
So we need some additional assumptions on the function $g$. 

In \cite{G}, \cite{GJ}, \cite{BaB}, \cite{BaB2}, \cite{BaP} the so called barycentric cost functions have been considered, i.e. the functions of the form $h(x, b(p))$, where
$Y = \mathbb R^d$ and $b(p)$ is the barycenter of a measure $p \in \mathcal P(\mathbb R^d)$. We generalize the notion of barycentric cost functions and consider cost functions of the form $h(x, g(p))$, where $g \colon \mathcal P(Y) \to \mathbb R^d$ is a finite-dimensional linear operator, $g(p) = \int_Y F(y) p(dy)$ for some Borel measurable function $F \colon Y \to \mathbb R^d$.  

Recall the definition of convex dominance for probability measures on $\mathbb R^d$. Let $\mu, \nu \in \mathcal P(\mathbb R^d)$. 
We say that the measure $\mu$ is dominated by $\nu$ in the convex order (and denote by $\mu \preceq_c \nu$) if both measures $\mu$ and $\nu$ have finite first moments and $\int_{\mathbb R^d} \varphi d\mu \le \int_{\mathbb R^d} \varphi d\nu$ for any convex function $\varphi \colon \mathbb R^d \to \mathbb R$ integrable with respect to the measure $\nu$.

\begin{theorem}\label{th_conv1}
Let $X$ and $Y$ be Souslin spaces, $\mu \in \mathcal P(X)$, $\nu \in \mathcal P(Y)$. Consider the nonlinear problem (\ref{gp}), where $g(p) = \int_Y F(y) p(dy)$, $F \colon Y \to \mathbb R^d$ is a Borel measurable function, integrable with respect to the measure $\nu$ and $h \colon X \times \mathbb R^d \to \mathbb R$ is a Borel measurable function. Then

(i) the infimum in the nonlinear Kantorovich problem (\ref{gp}) equals the infimum in the following Monge problem with convex dominance 
 \begin{equation} \label{conv1}
\inf_{\sigma \in \Pi(\mu, \nu)} \int_X h(x, g(\sigma^x)) \mu(dx) = \inf_{\scriptstyle T \colon X \to \mathbb R^d, \atop \scriptstyle \mu \circ T^{-1} \preceq_c \, \nu \circ F^{-1}} 
\int_X h(x, Tx) \mu(dx),
\end{equation}

(ii) there exists a minimum in the problem (\ref{gp}) if and only if there exists a minimum in the problem (\ref{conv1}).

\end{theorem}

For the proof of Theorem \ref{th_conv1} we exploit the following proposition which generalizes Strassen's theorem on existence of a martingale coupling. 

\begin{proposition}\label{prop_conv1} 
Let $Y$ be a Souslin space and let $F \colon Y \to \mathbb R^d$ be a Borel measurable function. 
Suppose that $\zeta \in \mathcal P(\mathbb R^d)$, $\nu \in \mathcal P(Y)$ and $\zeta \preceq_c \nu \circ F^{-1}$. Then there exists a measure $\pi \in \Pi(\zeta, \nu)$ such that
for conditional measures $\pi^u$ of the measure $\pi$ with respect to $\zeta$ we have 
$$\int_Y F(y) \pi^u(dy) = u \quad \mbox{for  $\zeta$-a.e. $u$}.$$
\end{proposition}

We prove Proposition \ref{prop_conv1} using the following theorem proved by Strassen in \cite{Strassen}. 

Let $S, T$ be complete separable metric spaces. Let $\hat \varphi$ and $\hat \psi$ be positive (not necessarily bounded from above) continuous functions on $S$ and $T$ respectively, bounded away from $0$. 
Set $$\hat \chi(s, t) =  \hat \varphi(s) + \hat \psi(t), \quad s \in S, t \in T. $$  
Let $\mathcal P_{\hat \chi}(S \times T)$ be the set of all Borel probability measures $\pi \in \mathcal P(S \times T)$ such that $\hat \chi(s, t)$ is $\pi$-integrable, endowed with the topology $\mathcal J$ generated by the functionals $\pi \mapsto \int \chi d\pi$ for all continuous functions
$\chi \colon S \times T \to \mathbb R$ with 
$$\sup\{|\chi(s, t)|/\hat \chi(s, t): s \in S, t \in T\} < \infty.$$

\begin{theorem}[\cite{Strassen}] \label{th_strassen}
Let $S, T$ be complete separable metric spaces. 
Let $\mu \in \mathcal P(S)$ and $\nu \in \mathcal P(T)$ be such that $\hat \varphi$ is $\mu$-integrable and $\hat \psi$ is $\nu$-integrable. Let $\Lambda$ be a $\mathcal J$-closed convex subset of $\mathcal P_{\hat \chi}(S \times T)$. 
The set $\Pi(\mu, \nu) \cap \Lambda$ is non-empty if and only if
$$
\int_S \varphi(s) \mu(ds) + \int_T \psi(t) \nu(dt) \le \sup\Bigl\{\int_{S \times T} (\varphi(s) + \psi(t)) \, \pi(ds dt): \pi \in \Lambda \Bigr\}
$$ 
for all continuous functions $\varphi \colon S \to \mathbb R$ with $\sup\{|\varphi(s)|/\hat \varphi(s): s \in S\} < \infty$ and 
continuous functions $\psi \colon T \to \mathbb R$ with
$\sup\{|\psi(t)|/\hat \psi(t): t \in T\} < \infty$. 
\end{theorem}

\begin{proof}[Proof of Proposition \ref{prop_conv1}]
We first consider the case where $Y$ is a complete separable metric space and $F$ is continuous. 
Denote by $\mathcal P_1(Y) \subset \mathcal P(Y)$ the set of all measures $\tilde \nu \in \mathcal P(Y)$ such that $|F|$ is $\tilde \nu$-integrable, $\mathcal P_1(\mathbb R^d) \subset \mathcal P(\mathbb R^d)$ --- the set of all measures 
$\tilde \zeta \in \mathcal P(\mathbb R^d)$ with a finite first moment, 
$\mathcal P_1(\mathbb R^d \times Y) \subset \mathcal P(\mathbb R^d \times Y)$ --- the set of all measures $\pi \in \mathcal P(\mathbb R^d \times Y)$ such that the function $|u| + |F(y)|$ is $\pi$-integrable. 
Endow $\mathcal P_1(\mathbb R^d \times Y)$ with the topology $\mathcal J$ generated by the functionals $\pi \mapsto \int \chi d\pi$ for all continuous functions $\chi \colon \mathbb R^d \times Y \to \mathbb R$ with 
$$\sup\{|\chi(u, y)|/(1 + |u| + |F(y)|): u \in \mathbb R^d, y \in Y\} < \infty.$$

Let 
$$
\Lambda = \Bigl\{\pi \in \mathcal P_1(\mathbb R^d \times Y): \int_Y F(y) \pi(du dy) = u \int_Y \pi(du dy)\Bigr\},
$$ 
i.e. $\Lambda$ is the set of all measures $\pi \in \mathcal P_1(\mathbb R^d \times Y)$ such that $(u, F(y))$ forms a martingale with respect to $\pi$. 
The set $\Lambda$ is convex and closed in the topology $\mathcal J$. Apply Theorem \ref{th_strassen} to prove that the set 
$\Pi(\zeta, \nu) \cap \Lambda$ is non-empty. 

Let $\varphi \colon \mathbb R^d \to \mathbb R$, $\psi \colon Y \to \mathbb R$ be continuous functions such that   
$\sup \{|\varphi(u)|/(1 + |u|):  u \in \mathbb R^d \} < \infty$ and $\sup\{|\psi(y)|/(1+|F(y)|): y \in Y\} < \infty$. 
Let us show that
$$
\int_{\mathbb R^d} \varphi(u) \zeta(du) + \int_Y \psi(y) \nu(dy) \le \sup\Bigl\{\int_{\mathbb R^d \times Y} (\varphi(u) + \psi(y)) \, \pi(du dy): \pi \in \Lambda\Bigr\}.
$$
Let $V$ be the set of all concave functions $f \colon \mathbb R^d \to \mathbb R$ such that $f(F(y)) \ge \psi(y)$ for all $y \in Y$. Set  $f_0(u) = \inf_{f \in V} f(u)$. Then the function $f_0$ is concave and  
\begin{multline*}
\int_{\mathbb R^d} \varphi(u) \zeta(du) + \int_Y \psi(y) \nu(dy) \le \int_{\mathbb R^d} \varphi(u) \zeta(du) + \int_Y f_0(F(y)) \nu(dy) = \\ =
 \int_{\mathbb R^d} \varphi(u) \zeta(du) + \int_{\mathbb R^d} f_0(u) \, \nu \circ F^{-1} (du).  
\end{multline*}
Since $\zeta \preceq_c \nu \circ F^{-1}$, we obtain that the last expression does not exceed
$$
\int_{\mathbb R^d} (\varphi(u) + f_0(u)) \zeta(du) \le \sup_{u \in \mathbb R^d} (\varphi(u) + f_0(u)).
$$
Thus
$$
\int_{\mathbb R^d} \varphi(u) \zeta(du) + \int_Y \psi(y) \nu(dy) \le \sup_{u \in \mathbb R^d} (\varphi(u) + f_0(u)).
$$
Let $r < \sup_{u \in \mathbb R^d} (\varphi(u) + f_0(u))$. Then there exists $u_0 \in \mathbb R^d$ such that $\varphi(u_0) + f_0(u_0) > r$. 
Let
$$
\Lambda_u = \Bigl\{p \in \mathcal P_1(Y): \int_Y F(y) p(dy) = u\Bigr\}.
$$ 
Set 
$$
f_1(u) = \sup \Bigl\{\int_Y \psi(y) p(dy): p \in \Lambda_u\Bigr\}.
$$ 
The function $f_1$ is concave and $f_1(F(y)) \ge \psi(y)$ for all $y \in Y$, since $\delta_y \in \Lambda_{F(y)}$. 
Therefore, $f_1(u) \ge f_0(u)$ for all $u \in \mathbb R^d$ and hence $r < \varphi(u_0) + f_1(u_0)$. 
By the definition of $f_1(u_0)$ there exists a measure $p \in \Lambda_{u_0}$ such that 
$$r < \varphi(u_0) + \int_Y \psi(y) p(dy).$$ 
Set $\pi = \delta_{u_0} \otimes p$. Then $\pi \in \Lambda$ because $\int_Y F(y) p(dy) = u_0$. 
Along with this, we have 
$$r < \int_{\mathbb R^d \times Y} (\varphi(u) + \psi(y)) \, \pi(du dy).$$

Therefore, by Theorem \ref{th_strassen} the set $\Pi(\zeta, \nu) \cap \Lambda$ is non-empty. 
Hence there exists a measure $\pi \in \Pi(\zeta, \nu)$ such that $(u, F(y))$ forms a martingale with respect to $\pi$. 

Let $Y$ be a complete separable metric space and let $F \colon Y \to \mathbb R^d$ be a Borel measurable function. Then there exists a sequence of bounded continuous functions $F_n \colon Y \to \mathbb R^d$ with $\|F_n - F\|_{L^1(\nu)} \to 0$. As proven above, there exist measures $\pi_n \in \Pi(\zeta, \nu)$ such that $(u, F_n(y))$ forms a martingale with respect to $\pi_n$. Since the set $\Pi(\zeta, \nu)$ is uniformly tight, we can pass to a subsequence $\pi_n$ converging weakly to a measure $\pi \in \Pi(\zeta, \nu)$. Let us show that $(u, F(y))$ forms a martingale  with respect to $\pi$. We prove that for any function $\varphi \in C_b(\mathbb R^d)$ with $|\varphi| \le 1$
$$
\int_{\mathbb R^d \times Y} \varphi(u) F(y) \, \pi(du dy) = \int_{\mathbb R^d} \varphi(u) u \, \zeta(du).
$$
For any $m \in \mathbb N$ we have
\begin{multline*}
\Bigl|\int_{\mathbb R^d \times Y} \varphi(u) F(y) \, \pi(du dy)  - \int_{\mathbb R^d \times Y} \varphi(u) F_m(y) \, \pi(du dy)\Bigr| \le \\ 
\le \int_{\mathbb R^d \times Y} |F(y) - F_m(y)| \, \pi(du dy) = \|F_m - F\|_{L^1(\nu)}.
\end{multline*}
Since $\pi_n$ converges weakly to $\pi$,
$$
\int_{\mathbb R^d \times Y} \varphi(u) F_m(y) \, \pi(du dy) = \lim_{n \to \infty} 
\int_{\mathbb R^d \times Y} \varphi(u) F_m(y) \, \pi_n(du dy).
$$
Furthermore,
\begin{multline*}
\Bigl|\int_{\mathbb R^d \times Y} \varphi(u) F_m(y) \, \pi_n(du dy) - \int_{\mathbb R^d \times Y} \varphi(u) F_n(y) \, \pi_n(du dy)\Bigr| \le \\ \le
\int_{\mathbb R^d \times Y} |F_m(y) - F_n(y)| \, \pi_n(du dy) = \|F_n - F_m\|_{L^1(\nu)} \le 
\|F_m - F\|_{L^1(\nu)} + \|F_n - F\|_{L^1(\nu)}.
\end{multline*}
Since $(u, F_n(y))$ forms a martingale with respect to $\pi_n$, we have
$$
\int_{\mathbb R^d \times Y} \varphi(u) F_n(y) \, \pi_n(du dy) = \int_{\mathbb R^d} \varphi(u) u \, \zeta(du).
$$
Therefore, 
$$
\Bigl|\int_{\mathbb R^d \times Y} \varphi(u) F(y) \, \pi(du dy) -  \int_{\mathbb R^d} \varphi(u) u \, \zeta(du) \Bigr| \le 
2\|F_m - F\|_{L^1(\nu)}. 
$$
Letting $m \to \infty$, we obtain that 
$$
\int_{\mathbb R^d \times Y} \varphi(u) F(y) \, \pi(du dy) = \int_{\mathbb R^d} \varphi(u) u \, \zeta(du).
$$

The statement also holds true in the case where $Y$ is a Souslin space, since Borel measures on Souslin spaces are isomorphic to Borel measures on complete separable metric spaces.

\end{proof}

\begin{proof}[Proof of Theorem \ref{th_conv1}]
Let $\sigma \in \Pi(\mu, \nu)$. Set $T(x) = g(\sigma^x)$ for all $x \in X$. Then $\int_X h(x, g(\sigma^x)) \, \mu(dx) = \int_X h(x, Tx) \, \mu(dx)$. Let us show that $\mu \circ T^{-1} \preceq_c \nu \circ F^{-1}$. Let $\varphi \colon \mathbb R^d \to \mathbb R$ be a convex function integrable with respect to $\nu \circ F^{-1}$. Then by Jensen's inequality we obtain 
\begin{multline*}
\int_{\mathbb R^d} \varphi(u) \, \mu \circ T^{-1} (du) = \int_X \varphi(g(\sigma^x)) \, \mu(dx) = \int_X \varphi \Bigl(\int_Y F(y) \sigma^x(dy) \Bigr) \, \mu(dx) \le \\
\le \int_X \int_Y \varphi(F(y)) \sigma^x(dy) \, \mu(dx) = \int_Y \varphi(F(y)) \nu(dy) = \int_{\mathbb R^d} \varphi(u) \, \nu \circ F^{-1}(du).
\end{multline*}
Therefore, $\mu \circ T^{-1} \preceq_c \nu \circ F^{-1}$. 

Conversely, let $T \colon X \to Y$ and $\mu \circ T^{-1} \preceq_c \nu \circ F^{-1}$. Set $\zeta = \mu \circ T^{-1}$. 
By Proposition \ref{prop_conv1} there exists a measure $\pi \in \Pi(\zeta, \nu)$ such that
for conditional measures $\pi^u$ of $\pi$ with respect to $\zeta$ we have 
$$\int_Y F(y) \pi^u(dy) = u \quad \mbox{for  $\zeta$-a.e. $u$}.$$

Set $\sigma^x = \pi^{T(x)}$ for all $x \in X$ and $\sigma (dx \, dy) = \sigma^x(dy) \mu(dx)$. Then $\sigma \in \Pi(\mu, \nu)$, since 
$$\int_X \sigma^x \mu(dx) = \int_X \pi^{T(x)} \mu(dx) = \int_{\mathbb R^d} \pi^u \zeta(du) = \nu.$$ 
Moreover,
$$g(\sigma^x) = \int_Y F(y) \pi^{T(x)}(dy) = T(x) \quad \mbox{for $\mu$-a.e. $x$}$$ 
and hence $\int_X h(x, g(\sigma^x)) \mu(dx) = \int_X h(x, Tx) \mu(dx)$. 
\end{proof}

A sufficient condition for the existence of a minimum in the Monge problem (\ref{conv1}) is that the cost function $h$ has the form
$h(x, u) = \psi(x - u)$, where $\psi \colon \mathbb R^d \to \mathbb R$ is a strictly convex function (and $X = \mathbb R^d$), but in this case the function $h(x, g(p))$ is convex in $p$ so it is known that there exists a minimum in the nonlinear Kantorovich problem (\ref{gp}) by Theorem \ref{th_xp}.

\begin{remark}
\rm
Let $X = \mathbb R$ and let $Y$ be a Souslin space. Let $F \colon Y \to \mathbb R$ be a Borel measurable function and 
$h(x, u) = \psi(x - u)$ for all $x, u\in \mathbb R$, where $\psi \colon \mathbb R \to \mathbb R$ is a convex function. Let $\uline{T}$ be the weak monotone rearrangement of measures $\mu$ and $\nu \circ F^{-1}$ (see \cite{BaB2}). 
Set $\zeta = \mu \circ \uline{T}^{-1}$. By Proposition \ref{conv1} take a measure $\pi \in \Pi(\zeta, \nu)$ such that 
$\int_Y F(y) \pi^u(dy) = u$ for $\zeta$-a.e. $u$. Then the measure $\sigma = \pi^{\uline{T}(x)}(dy) \mu(dx)$ delivers a minimum in the problem (\ref{gp}).

Let $X = \mathbb R^d$ and let $Y$ be a Souslin space. Let $F \colon Y \to \mathbb R^d$ be a Borel measurable function and 
$h(x, u) = |x - u|^2$ for all $x, u \in \mathbb R^d$. By Brenier-Strassen theorem (see \cite{GJ}) take a convex function $\varphi$ with $1$-Lipschitz gradient such that $\mu \circ (\nabla \varphi)^{-1} \preceq_c \nu \circ F^{-1}$. 
Set $\zeta = \mu \circ (\nabla \varphi)^{-1}$ and by Proposition \ref{conv1} take a measure $\pi \in \Pi(\zeta, \nu)$ such that $\int_Y F(y) \pi^u(dy) = u$ for $\zeta$-a.e. $u$. 
Then the measure $\sigma = \pi^{\nabla \varphi(x)}(dy) \mu(dx)$ delivers a minimum in the problem (\ref{gp}).
\end{remark}

Consider the following examples where the infimum in the problem (\ref{conv1}) (and, therefore, in the problem (\ref{gp})) is not attained. 

\begin{example}\label{ex1}
\rm
Let $X = [0, 1]$, $Y = [0, 1] \times [-1, 1]$, $\mu = \lambda$ is Lebesgue measure on the interval $[0, 1]$, the measure $\nu$ equals half-sum of Lebesgue measures on the intervals $[0, 1] \times \{-1\}$ and $[0, 1] \times \{1\}$ respectively. 
Consider the nonlinear Kantorovich problem with barycentric cost function $h(x, b(p))$, where 
$$h(x, y) = (x - y_1)^2 + 1 - y_2^2, \quad x \in \mathbb R, \, y = (y_1, y_2) \in \mathbb R^2.$$ 

We show that the infimum in the problem (\ref{conv1}) equals $0$, but it is not attained. 
Suppose that there exists a mapping $T \colon X \to Y$ such that 
$$\mu \circ T^{-1} \preceq_c \nu \quad \mbox{and} \quad \int_X h(x, Tx) \mu(dx) = 0.$$ 
Then $Tx \in \{(x, -1), (x, 1)\}$ for $\mu$-a.e. $x$. 
Let $A = \{x \in [0, 1]: Tx = (x, -1)\}.$ Then 
$$\mu \circ T^{-1} = I_A \lambda|_{[0, 1] \times \{-1\}} + I_{[0, 1] \setminus A} \lambda|_{[0, 1] \times \{1\}}.$$ 
Since $\mu \circ T^{-1} \preceq_c \nu$, we obtain that for any convex function $\varphi \colon [0, 1] \to \mathbb R$ 
$$\int_A \varphi(y) dy \le \frac{1}{2} \int_0^1 \varphi(y) dy$$
and 
$$\int_{[0, 1] \setminus A} \varphi(y) dy \le \frac{1}{2} \int_0^1 \varphi(y) dy.$$
This implies that $\int_A \varphi(y) dy = \frac{1}{2} \int_0^1 \varphi(y) dy$ for all convex functions $\varphi \colon [0, 1] \to \mathbb R$. Therefore, $I_A = \frac{1}{2}$ a.e., which is a contradiction. 
Thus there is no mapping $T \colon X \to Y$ with $\mu \circ T^{-1} \preceq_c \nu$ and $\int_X h(x, Tx) \mu(dx) = 0$.

We show that there is a sequence of mappings $T_n \colon X \to Y$ such that 
$$\mu \circ T_n^{-1} \preceq_c \nu \quad \mbox{and} \quad \int_X h(x, T_n x) \mu(dx) \to 0.$$ 
Set 
$$T_n(x) = (2x - (2k)/2^n, 1) \quad \mbox{if  } x \in [(2k)/2^n, (2k+1)/2^n),$$
$$T_n(x) = (2x - (2k+2)/2^n, -1) \quad \mbox{if  } x \in [(2k+1)/2^n, (2k+2)/2^n),$$
where $k \in \{0, \dots, 2^{n - 1} - 1\}$. Then $\mu \circ T_n^{-1} = \nu$ and
$$
\int_X h(x, T_n x) \mu(dx) = \int_0^1 (x - (T_n x)_1)^2 dx \le 1/2^{n}.
$$
Therefore, the infimum in the problem (\ref{conv1}) is not attained.  
\end{example}

\begin{example}\label{ex2}
\rm
Let $X = [0, 1]$, $Y = [-1, 1]$, $\mu$ is Lebesgue measure on $[0, 1]$, $\nu$ is normalized Lebesgue measure on $[-1, 1]$. Consider the nonlinear Kantorovich problem with barycentric cost function $h(x, b(p))$, 
where $h(x, y) = (x^2 - y^2)^2$.
We show that the infimum in the problem (\ref{conv1}) equals $0$, but it is not attained. 
Suppose that there exists a mapping $T \colon X \to Y$ such that 
$$\mu \circ T^{-1} \preceq_c \nu \quad \mbox{and} \quad \int_X h(x, Tx) \mu(dx) = 0.$$ 
Then $Tx \in \{x, -x\}$ for $\mu$-a.e. $x$. 
Let $A = \{x \in [0, 1]: Tx = x\}$. Then 
$$\mu \circ T^{-1} = I_A + I_{-[0, 1] \setminus A}.$$
Set $\varphi(x) = 0$ for $x \le 0$ and $\varphi(x) = x^k$ for $x > 0$, where $k \in \mathbb N$. 
The function $\varphi$ is convex, therefore by the property $\mu \circ T^{-1} \preceq_c \nu$ we obtain that
$$
\int_A x^k dx \le \frac{1}{2} \int_0^1 x^k dx.
$$
Similarly we get 
$$
\int_{[0, 1] \setminus A} x^k dx \le \frac{1}{2} \int_0^1 x^k dx.
$$
Therefore, $\int_A x^k dx = \frac{1}{2} \int_0^1 x^k dx$ for any $k \in \mathbb N$. 
Thus the measures $x I_A dx$ and $\frac{1}{2} x I_{[0, 1]} dx$ are equal. Then $I_A = \frac{1}{2}$ a.e., which is a contradiction. 

Likewise as in Example \ref{ex1}, we can take a sequence of mappings $T_n \colon X \to Y$ such that 
$\mu \circ T_n^{-1} \preceq_c \nu$ and $\int_X h(x, T_n x) \mu(dx) \to 0$. 

This example also shows that the statement of Theorem \ref{th_psi} is not true if in the case \rm{(ii)} of Theorem \ref{th_psi} we do not assume that the convex function $\Psi$ is increasing. 
Indeed, put in Theorem \ref{th_psi} $\Psi(v) = v^2$, $f(x) = x^2$, $g(p) = b(p)^2$.  
The sets $\{p \in \mathcal P(Y): g(p) \le c\}$ are convex for every $c \in \mathbb R$ but the function $\Psi$ is not increasing. 
As we have shown above, there is no minimum in the nonlinear Kantorovich problem with the cost function $\Psi(g(p) - f(x))$. 
\end{example}

We also provide an infinite-dimensional analogue of Theorem \ref{th_conv1}. 

Let $U$ be a separable Banach space. Denote by $\|\cdot\|_{U}$ the norm on $U$. We say that a measure $\mu \in \mathcal P(U)$ has the barycenter $a \in U$ if 
$$l(a) = \int_U l(u) \mu(du)$$ for every continuous linear functional $l \in U^*$. 
It is known that any measure $\mu \in \mathcal P(U)$ with $\int_U \|u\|_U \mu(du) < \infty$ has a barycenter. 

Define the convex dominance for probability measures on $U$. Let $\mu, \nu \in \mathcal P(U)$. We say that the measure $\mu$ is dominated by $\nu$ in the convex order (and denote by $\mu \preceq_c \nu$) if both measures $\mu$ and $\nu$ have barycenters and $\int_U \varphi d\mu \le \int_U \varphi d\nu$ for any lower semicontinuous convex function $\varphi \colon U \to \mathbb R$ integrable with respect to the measure $\nu$. 

\begin{theorem}\label{th_conv2}
Let $X$ and $Y$ be Souslin spaces, $\mu \in \mathcal P(X)$, $\nu \in \mathcal P(Y)$. Consider the nonlinear problem (\ref{gp}), where $g(p) = \int_Y F(y) p(dy)$, $F \colon Y \to U$ is a Borel measurable mapping, $\|F\|_U$ is integrable with respect to the measure $\nu$ and $h \colon X \times U \to \mathbb R$ is a Borel measurable function. Then

(i) the infimum in the nonlinear Kantorovich problem (\ref{gp}) equals the infimum in the following Monge problem with convex dominance 
\begin{equation} \label{conv2}
\inf_{\sigma \in \Pi(\mu, \nu)} \int_X h(x, g(\sigma^x)) \mu(dx) = 
\inf_{\scriptstyle T \colon X \to U, \atop \scriptstyle \mu \circ T^{-1} \preceq_c \, \nu \circ F^{-1}} \int_X h(x, Tx) \mu(dx),
\end{equation}

(ii) there exists a minimum in the problem (\ref{gp}) if and only if there exists a minimum in the problem (\ref{conv2}).

\end{theorem}

For the proof of Theorem \ref{th_conv2} we use the following infinite-dimensional analogue of Proposition \ref{prop_conv1}. 

\begin{proposition}\label{prop_conv2}
Let $U$ be a separable Banach space, let $Y$ be a Souslin space and let $F \colon Y \to U$ be a Borel measurable mapping. 
Suppose that $\zeta \in \mathcal P(U)$, $\nu \in \mathcal P(Y)$, $\|F\|_U$ is integrable with respect to the measure $\nu$ and $\zeta \preceq_c \nu \circ F^{-1}$. Then there exists a measure $\pi \in \Pi(\zeta, \nu)$ such that for conditional measures $\pi^u$ of the measure $\pi$ with respect to $\zeta$ we have 
$$\int_Y F(y) \pi^u(dy) = u \quad \mbox{for  $\zeta$-a.e. $u$}.$$
\end{proposition}

\begin{proof}[Proof of Proposition \ref{prop_conv2}]

We first consider the case where $Y$ is a complete separable metric space and $F \colon Y \to U$ is a surjective continuous mapping.   
Denote by $\mathcal P_1(Y) \subset \mathcal P(Y)$ the set of all measures $\tilde \nu \in \mathcal P(Y)$ such that $\|F\|_{U}$ is $\tilde \nu$-integrable, $\mathcal P_1(U) \subset \mathcal P(U)$ --- the set of all measures 
$\tilde \zeta \in \mathcal P(U)$ such that $\|u\|_U$ is $\tilde \zeta$-integrable,  
$\mathcal P_1(U \times Y) \subset \mathcal P(U \times Y)$ --- the set of all measures $\pi \in \mathcal P(U \times Y)$ such that the function $\|u\|_U + \|F(y)\|_U$ is $\pi$-integrable. 
Endow $\mathcal P_1(U \times Y)$ with the topology $\mathcal J$ generated by the functionals $\pi \mapsto \int \chi d\pi$ for all continuous functions $\chi \colon U \times Y \to \mathbb R$ with 
$$\sup\{|\chi(u, y)|/(1 + \|u\|_U + \|F(y)\|_U): u \in U, y \in Y\} < \infty.$$

Let
\begin{multline*}
\Lambda = \{\pi \in \mathcal P_1(U \times Y): \int_{U \times Y} q(u) l(F(y)) \, \pi(du dy) = \int_{U \times Y} q(u) l(u) \, \pi(du dy) \\ \forall q \in C_b(U), l \in U^*\}, 
\end{multline*}
i.e. $\Lambda$ is the set of all measures $\pi \in \mathcal P_1(U \times Y)$ such that $(u, F(y))$ forms a martingale with respect to $\pi$. 

The set $\Lambda$ is convex and closed in the topology $\mathcal J$. Apply Theorem \ref{th_strassen} to prove that the set 
$\Pi(\zeta, \nu) \cap \Lambda$ is non-empty. 

Let $\varphi \colon U \to \mathbb R$, $\psi \colon Y \to \mathbb R$  be continuous functions such that   
$\sup \{|\varphi(u)|/(1 + \|u\|_U): u \in U\} < \infty$ and $\sup \{|\psi(y)|/(1+\|F(y)\|_U): y \in Y\} < \infty$. 
Let us show that
$$
\int_U \varphi(u) \zeta(du) + \int_Y \psi(y) \nu(dy) \le \sup\Bigl\{\int_{U \times Y} (\varphi(u) + \psi(y)) \, \pi(du dy): \pi \in \Lambda \Bigr\}.
$$

Let $V$ be the set of all upper semicontinuous concave functions $f \colon U \to \mathbb R$ such that $f(F(y)) \ge \psi(y)$ for all $y \in Y$. Set $f_0(u) = \inf_{f \in V} f(u)$. Then $f_0$ is an upper semicontinuous concave function. We have
\begin{multline*}
\int_U \varphi(u) \zeta(du) + \int_Y \psi(y) \nu(dy) \le \int_Y \varphi(u) \zeta(du) + \int_Y f_0(F(y)) \nu(dy) = \\ =
 \int_U \varphi(u) \zeta(du) + \int_U f_0(u) \, \nu \circ F^{-1} (du). 
\end{multline*}
Since $\zeta \preceq_c \nu \circ F^{-1}$, the last expression does not exceed
$$
\int_U (\varphi(u) + f_0(u)) \zeta(du) \le \sup_{u \in U} (\varphi(u) + f_0(u)).
$$
Let $r < \sup_{u \in U} (\varphi(u) + f_0(u))$.
Then there exists $u_0 \in U$ such that $\varphi(u_0) + f_0(u_0) > r$. 
Let
$$
\Lambda_u = \Bigl\{p \in \mathcal P_1(Y): \int_Y F(y) p(dy) = u\Bigr\}. 
$$ 
Set 
$$
f_1(u) = \sup \Bigl\{\int_Y \psi(y) p(dy): p \in \Lambda_u\Bigr\}.
$$ 

The function $f_1$ is concave and upper semicontinuous (moreover, $f_1$ is Lipschitz). 
Indeed, let us prove that $|f_1(u_2) - f_1(u_1)| \le C \|u_2 - u_1\|_U$ for all $u_1, u_2 \in U$, where 
$$C = \sup \{|\psi(y)|/(1+\|F(y)\|_U): y \in Y\}.$$
Let $p_1 \in \Lambda_{u_1}$. Show that
$$f_1(u_2) \ge \int_Y \psi(y) p_1(dy) - C\|u_2 - u_1\|_U.$$ 
Fix $\alpha \in (0, 1)$.  By surjectivity of $F$ take $y_0 \in Y$ with $F(y_0) = (u_2 - \alpha u_1)/(1 - \alpha)$ and set 
$p_2 = \alpha p_1 + (1 - \alpha) \delta_{y_0}$. Then
$$
\int_Y F(y) p_2(dy) = \alpha \int_Y F(y) p_1(dy) + (1 - \alpha) F(y_0) = \alpha u_1 + (1 - \alpha) F(y_0) = u_2,
$$ 
i.e. $p_2 \in \Lambda_{u_2}$. Therefore, 
\begin{multline*}
f_1(u_2) \ge \int_Y \psi(y) p_2(dy) = \alpha \int_Y \psi(y) p_1(dy) + (1 - \alpha) \psi(y_0) \ge \alpha \int_Y \psi(y) p_1(dy) - \\ - C (1 - \alpha) (1 + \|F(y_0)\|_U) = \alpha \int_Y \psi(y) p_1(dy) - C (1 - \alpha) - C \|u_2 - \alpha u_1\|_U.
\end{multline*}
Letting $\alpha \to 1$, we obtain the inequality $f_1(u_2) \ge \int_Y \psi(y) p_1(dy) - C \|u_2 - u_1\|_U$. 
Thus, $f_1$ is Lipschitz with constant $C$. 

Furthermore, $f_1(F(y)) \ge \psi(y)$ for all $y \in Y$, since $\delta_y \in \Lambda_{F(y)}$. 
Therefore, $f_1(u) \ge f_0(u)$ for all $u \in U$ and hence $r < \varphi(u_0) + f_1(u_0)$. 
By the definition of $f_1(u_0)$ there exists a measure $p \in \Lambda_{u_0}$ such that 
$$r < \varphi(u_0) + \int_Y \psi(y) p(dy).$$ 
Set $\pi = \delta_{u_0} \otimes p$. Then $\pi \in \Lambda$ because $\int_Y F(y) p(dy) = u_0$. 
Along with this, we have 
$$r < \int_{U \times Y} (\varphi(u) + \psi(y)) \, \pi(du dy).$$
Therefore, by Theorem \ref{th_strassen} the set $\Pi(\zeta, \nu) \cap \Lambda$ is non-empty. 
Hence there exists a measure $\pi \in \Pi(\zeta, \nu)$ such that $(u, F(y))$ forms a martingale with respect to $\pi$.

Let $Y$ be a complete separable metric space and let $F \colon Y \to U$ be a continuous mapping, $F$ is not surjective. 
Take the disjoint union $\tilde Y = Y \sqcup U$, then $\tilde Y$ is a complete separable metric space. 
Define the mapping $\tilde F \colon \tilde Y \to U$: $\tilde F(\tilde y) = F(\tilde y)$ for all $\tilde y \in Y$ and $\tilde F(\tilde y) = \tilde y$ for all $\tilde y \in U$. The mapping $\tilde F$ is continuous on $\tilde Y$ and $\tilde F(\tilde Y) = U$. 
Extend $\nu$ to a measure on $\tilde Y$.
Then there exists a measure $\pi \in \mathcal P(U \times \tilde Y)$ such that $\pi \in \Pi(\zeta, \nu)$ and $(u, \tilde F(\tilde y))$ forms a martingale with respect to $\pi$. Since $\nu$ is concentrated on $Y$, we obtain that $\pi \in \mathcal P(U \times Y)$
and $(u, F(y))$ forms a martingale with respect to $\pi$. 

Let $Y$ be a complete separable metric space and let $F \colon Y \to U$ be a Borel measurable mapping. 
Take a sequence of bounded continuous mappings $F_n \colon Y \to U$ with $\|\|F_n - F\|_U\|_{L^1(\nu)} \to 0$. 
Indeed, set $\tilde F_k(y) = F(y)$ if $\|F(y)\|_U \le k$ and $\tilde F_k(y) = 0$ otherwise, where $k \in \mathbb N$. 
Then $\|\|\tilde F_k - F\|_U\|_{L^1(\nu)} \to 0$ as $k \to \infty$. 
By Lusin's theorem (see \cite[Theorem~7.1.13]{B07}) for any $n \in \mathbb N$ we can take a continuous mapping $\tilde F_{k, n} \colon Y \to U$ such that
$\nu(y \in Y: \tilde F_{k, n}(y) \neq \tilde F_k(y)) < 1/n$ and $\sup_{y \in Y} \|\tilde F_{k, n}(y)\|_U \le k$. Then $\|\|\tilde F_{k, n} - \tilde F_k\|_U\|_{L^1(\nu)} \le 2k/n$. The mappings $\tilde F_{k, n}$ are bounded continuous and we can take a sequence $\tilde F_{k, n_k}$ such that $\|\|\tilde F_{k, n_k} - F\|_U\|_{L^1(\nu)} \to 0$. 

As proven above, there exist measures $\pi_n \in \Pi(\zeta, \nu)$ such that $(u, F_n(y))$ forms a martingale with respect to $\pi_n$. Since the set $\Pi(\zeta, \nu)$ is uniformly tight, we can pass to a subsequence $\pi_n$ converging weakly to a measure $\pi \in \Pi(\zeta, \nu)$. Let us show that $(u, F(y))$ forms a martingale  with respect to $\pi$. We prove that for any function $\varphi \in C_b(\mathbb R^d)$ with $|\varphi| \le 1$ and for any continuous linear functional $l \in U^*$ 
with $\|l\|_{U^*} \le 1$
$$
\int_{U \times Y} \varphi(u) l(F(y)) \, \pi(du dy) = \int_U \varphi(u) l(u) \, \zeta(du).
$$
For any $m \in \mathbb N$ we have
\begin{multline*}
\Bigl|\int_{U \times Y} \varphi(u) l(F(y)) \, \pi(du dy)  - \int_{U \times Y} \varphi(u) l(F_m(y)) \, \pi(du dy)\Bigr| \le \\ 
\le \int_{U \times Y} \|F(y) - F_m(y)\|_U \, \pi(du dy) = \|\|F_m - F\|_U\|_{L^1(\nu)}.
\end{multline*}
Since $\pi_n$ converges weakly to $\pi$,
$$
\int_{U \times Y} \varphi(u) l(F_m(y)) \, \pi(du dy) = \lim_{n \to \infty} 
\int_{U \times Y} \varphi(u) l(F_m(y)) \, \pi_n(du dy).
$$
Furthermore,
\begin{multline*}
\Bigl|\int_{U \times Y} \varphi(u) l(F_m(y)) \, \pi_n(du dy) - \int_{U \times Y} \varphi(u) l(F_n(y)) \, \pi_n(du dy)\Bigr| \le \\ \le
\int_{U \times Y} \|F_m(y) - F_n(y)\|_U \, \pi_n(du dy) \le 
\|\|F_m - F\|_U\|_{L^1(\nu)} + \|\|F_n - F\|_U\|_{L^1(\nu)}.
\end{multline*}
Since $(u, F_n(y))$ forms a martingale with respect to $\pi_n$, we have
$$
\int_{U \times Y} \varphi(u) l(F_n(y)) \, \pi_n(du dy) = \int_U \varphi(u) l(u) \, \zeta(du).
$$
Therefore, 
$$
\Bigl|\int_{U \times Y} \varphi(u) l(F(y)) \, \pi(du dy) -  \int_{U} \varphi(u) l(u) \, \zeta(du) \Bigr| \le 
2\|\|F_m - F\|_U\|_{L^1(\nu)}. 
$$
Letting $m \to \infty$, we obtain that 
$$
\int_{U \times Y} \varphi(u) l(F(y)) \, \pi(du dy) = \int_{U} \varphi(u) l(u) \, \zeta(du).
$$

The statement also holds true in the case where $Y$ is a Souslin space, since Borel measures on Souslin spaces are isomorphic to Borel measures on complete separable metric spaces.
 
Note that the condition
$$
\int_{U \times Y} \varphi(u) l(F(y)) \, \pi(du dy) = \int_{U} \varphi(u) l(u) \, \zeta(du) \quad 
\forall \varphi \in C_b(U), \, l \in U^*
$$
implies that
$$\int_Y F(y) \pi^u(dy) = u \quad \mbox{for  $\zeta$-a.e. $u$}.$$
Indeed, for any $l \in U^*$ we obtain that
$$
\int_Y l(F(y)) \pi^u(dy) = l(u) \quad \mbox{for  $\zeta$-a.e. $u$}.
$$
Take a set $\{l_n, n \in \mathbb N\} \subset U^*$ such that any $l \in U^*$ is a limit of a subsequence of $\{l_n\}$ in the weak-$*$ topology on $U^*$
(such a set exists since $U$ is separable). Then for $\zeta$-a.e. $u$ we have
$$
\int_Y l_n(F(y)) \pi^u(dy) = l_n(u) \quad \forall n \in \mathbb N.
$$ 
The last relation implies that $\int_Y l(F(y)) \pi^u(dy) = l(u)$ for all $l \in U^*$. 
Indeed, for any $l \in U^*$ take a subsequence $\{l_{n_k}\}$ which converges weakly-$*$ to $l$. 
The weakly-$*$ convergent sequence is bounded: $\|l_{n_k}\|_{U^*} \le C$ for some $C > 0$.
Then $l_{n_k}(F(y)) \to l(F(y))$ for all $y \in Y$ and $|l_{n_k}(F(y))| \le C \|F(y)\|_U$. 
Thus we obtain the desired equality passing to the limit as $k \to \infty$ and using the fact that
$\|F(y)\|_U$ is $\pi^u$-integrable for $\zeta$-a.e. $u$. (because $\|F(y)\|_U$ is $\nu$-integrable).

\end{proof}

\begin{proof}[Proof of Theorem \ref{th_conv2}]
Let $\sigma \in \Pi(\mu, \nu)$. Set $T(x) = g(\sigma^x)$ for all $x \in X$.  We show that $\mu \circ T^{-1} \preceq_c \nu \circ F^{-1}$. Let $\varphi \colon U \to \mathbb R$ be a lower semicontinuous convex function. Then by Jensen's inequality 
\begin{multline*}
\int_U \varphi(u) \, \mu \circ T^{-1} (du) = \int_X \varphi \Bigl(\int_Y F(y) \sigma^x(dy) \Bigr) \mu(dx) \le \\
\le \int_X \int_Y \varphi(F(y)) \sigma^x(dy) \mu(dx) = \int_U \varphi(u) \, \nu \circ F^{-1}(du).
\end{multline*}
Therefore, $\mu \circ T^{-1} \preceq_c \nu \circ F^{-1}$. 

Conversely, let $T \colon X \to Y$ and $\mu \circ T^{-1} \preceq_c \nu \circ F^{-1}$. Denote $\zeta = \mu \circ T^{-1}$. 
By Proposition \ref{prop_conv2} we take a measure $\pi \in \Pi(\zeta, \nu)$ such that
for conditional measures $\pi^u$ of $\pi$ with respect to $\zeta$ we have 
$$\int_Y F(y) \pi^u(dy) = u \quad \mbox{for  $\zeta$-a.e. $u$}.$$
Set $\sigma^x = \pi^{T(x)}$, $\sigma (dx \, dy) = \sigma^x(dy) \mu(dx)$. Similarly to the proof of Theorem \ref{th_conv1} 
we obtain that $\sigma \in \Pi(\mu, \nu)$ and $g(\sigma^x) = T(x)$ for $\mu$-a.e. $x$. 

Therefore, the infima in the problems (\ref{gp}) and (\ref{conv2}) coincide.
 
\end{proof}


\end{document}